\documentclass[12pt, reqno]{amsproc}
\usepackage{amsmath, amsthm, amscd, amsfonts, amssymb, graphicx}
\usepackage{geometry}
\usepackage{cite}
\usepackage{calc}
\usepackage{color,hyperref}
\definecolor{darkblue}{rgb}{0.0,0.0,0.3}
\hypersetup{colorlinks,breaklinks,
            linkcolor=red,urlcolor=pink,
            anchorcolor=darkblue,citecolor=blue}
\newcommand{\makeheading}[2][]%
        {\hspace*{-\marginparsep minus \marginparwidth}%
         \begin{minipage}[t]{\textwidth+\marginparwidth+\marginparsep}%
             {\large \bfseries #2 \hfill #1}\\[-0.15\baselineskip]%
                 \rule{\columnwidth}{1pt}%
         \end{minipage}}

\newtheorem{theorem}{\bf{Theorem}}[section] 
\newtheorem{lemma}[theorem]{\bf{Lemma}}     

\theoremstyle{definition}

\title[Essential norm of the extensions of Stevi\'{c}-Sharma operator]{Essential norm of the extensions of Stevi\'{c}-Sharma operator on
some spaces of analytic functions}
\author[Hassanlou]{Mostafa Hassanlou$^*$}
\address{Engineering Faculty of Khoy, Urmia University of Technology, Urmia, Iran}
\email{m.hassanlou@urmia.ac.ir}
\author[Gissy]{Hussain Gissy$^*$}
\address{Department of Mathematics, Faculty of Science, Jazan University, P.O. Box 2097, Jazan 45142, Kingdom of Saudi Arabia.}
\email{hgissy@jazanu.edu.sa}
\keywords{Stevi\'{c}-Sharma operator, Essential norm, Bloch space.\\
\indent $^{*}$ Corresponding author}
\begin{document}

\vspace{1cm} \setcounter{page}{1} \thispagestyle{empty}
\maketitle
\begin{abstract}
In this paper, we consider two extensions  of Stevi\'{c}-Sharma operator and find estimations for the essential norm of
them from $\mathcal{Q}_K (p,q)$ and $H^{\infty}$ into weighted Bloch spaces.
\end{abstract}
\section{Introduction}
The essential norm of an operator is the distance of the operator  from the space of compact operators. Let $f$ belongs to the space of the analytic functions on the unit disc. Suppose that $\varphi$ is an analytic self
map of the unit disc. The composition operator induced by $\varphi$ is denoted by $C_{\varphi}$, and is defined by $C_{\varphi}f = f \circ \varphi$. The weighted
composition operator $u C_{\varphi}$, $u$ is an analytic function on the unit disc, is defined by $u C_{\varphi} f = u f \circ \varphi$. Through the decades, there has
been an interest on investigating the (weighted) composition operators on spaces of functions, which started somehow after the excellent book of Cowen and MacCluer, \cite{CM}. There are
several operators related to these operators  using derivative or integral such as $u C_{\varphi} D$, $D$ is the differentiation operator, $u C_{\varphi} D f = u f' \circ \varphi$. In \cite{SS1,SS2}, the Stevi\'{c}-Sharma operator was introduced and it can be viewed as sum of $u C_{\varphi}$ and $u C_{\varphi} D $ which is of importance due to the relation with
topological structure of the spaces of such operators.

Our aim in this paper is to estimating the essential norm of  two extensions of Stevi\'{c}-Sharma operator  on some well-known spaces of analytic functions, $H^{\infty}$ and
$\mathcal{Q}_K (p,q)$. As results we obtain compactness criteria for these operators which are proved before.
\section{Background Material}
Let $\mathbb{D}$ be unit disk $\{z\in\mathbb{C}:|z|<1\}$, $\mathcal{H}(\mathbb{D})$ be the space of all analytic functions on $\mathbb{D}$ and $\mathcal{S}(\mathbb{D})$ be
the set of all analytic self-maps of $\mathbb{D}$. Let $H^{\infty}$ be the space of
bounded analytic functions on $\mathbb{D}$ with sup-norm, $\| f \|_{\infty} = \sup_{z \in \mathbb{D}} |f(z)|$.

A positive continuous function $\phi$ on $[0,1)$ is called normal if there are two constants $b > a > 0$  and $\delta \in [0,1)$ such that
\begin{itemize}
  \item [(i)] $\frac{\phi(r)}{(1-r)^a}$ is decreasing on $[\delta,1)$ and
  $\frac{\phi(r)}{(1-r)^a} \downarrow 0$
  \item [(ii)] $\frac{\phi(r)}{(1-r)^b}$ is increasing  on $[\delta,1)$ and
  $\frac{\phi(r)}{(1-r)^b} \uparrow \infty$
\end{itemize}
as $r \rightarrow 1$. An example of normal functions is $(1-r^2)^{\alpha}$, $\alpha >0$. Let $\mu$ be a positive function that is normal and radial, $\mu(z) = \mu(|z|)$. The weighted Bloch space $\mathcal{B}_{\mu}$ is the space of all analytic functions $f$ for which
$$ \sup_{z \in \mathbb{D}} \mu(z) |f'(z)| < \infty, $$
and it becomes a Banach space if we add $|f(0)|$ to the above statement. The $\alpha$-Bloch space $\mathcal{B}^{\alpha}$ is obtained when $\mu(z) = (1-|z|^2)^{\alpha}$ for $\alpha > 0$, see \cite{zhu1}. Also for $f \in \mathcal{B}^{\alpha}$,
\begin{equation}
\| f \|_{\mathcal{B}^{\alpha}} \approx
|f'(0)| + \cdots + |f^{(n)}(0)| + \sup_{z\in \mathbb{D}} (1-|z|^2)^{\alpha +n} |f^{(n+1)}(z)|, \ \ \ n \in \mathbb{N}.
\end{equation}
Let $p>0$, $q>-2$ and $ K: [0,\infty) \longrightarrow [0,\infty) $ be a nondecreasing continuous function. The space $\mathcal{Q}_K(p,q)$ consists of all $f \in \mathcal{H}(\mathbb{D})$ such that
\begin{equation*}
  \|f\|_{\mathcal{Q}_K (p,q)}^p=|f(0)|+\sup_{\xi\in\mathbb{D}}\int_\mathbb{D}|f'(z)|^p\left(1-|z|^2\right)^qK(g(z,\xi))\ dA(z)<\infty,
\end{equation*}
where $dA$ is the normalized Lebesgue area measure in $\mathbb{D}$, $g(z,\xi)=\log\frac{1}{\left|\varphi_\xi(z)\right|}$ is the Green function and $\varphi_\xi(z)=\frac{\xi-z}{1-\bar{\xi}z}$. For $p\geq 1$, $\mathcal{Q}_K(p,q)$ with the norm $\|f\|_{\mathcal{Q}_K(p,q)}$ becomes a Banach space. For more information on these spaces and operators on them, see \cite{kot,Liu,MHV,M,P,Ren}.  Following \cite{Wz}, we assume that the following condition holds
\begin{equation*}
  \int_0^1\left(1-r^2\right)^q\ K(-\log\ r)r\ dr<\infty,
\end{equation*}
since otherwise $\mathcal{Q}_K(p,q)$ consists only of constant functions. For $f \in \mathcal{Q}_K(p,q)$ we have $f \in \mathcal{B}^{\frac{q+2}{p}}$ and
\begin{equation}\label{r40}
  \|f \|_{\mathcal{B}^{\frac{q+2}{p}}} \leq  \|f \|_{\mathcal{Q}_K(p,q)}.
\end{equation}
A function $h$ with derivative $|h'(z)|^p = |1-z|^{-q-2}$ represents an extremal growth in $ \mathcal{B}^{\frac{q+2}{p}}$ and by Lemma 2.7 \cite{kot}, there exist parameters $p,q$ and $K$
such that $h \in \mathcal{Q}_K(p,q) \subsetneq \mathcal{B}^{\frac{q+2}{p}}$. Also taking $K(t)= t^s$, $0 < s<1$ and $q> -s-1$ imply that
$\mathcal{Q}_K(p,q)$ coincides with the non-trivial $F(p,q,s)$ space which includes some other well-known spaces of analytic functions, see \cite{zhao}. For some researches regarding $\mathcal{Q}_K(p,q)$ space, one can refer to \cite{ZX,Liu,M,Ren,Wz,Y,Y}.

For $\psi_1, \psi_2 \in \mathcal{H}(\mathbb{D})$ and $\varphi \in \mathcal{S}(\mathbb{D})$, the Stevi\'{c}-Sharma operator is defined by
$$ (T_{\psi_1, \psi_2, \varphi} f)(z) = \psi_1 (z) f(\varphi(z)) +\psi_2 (z) f'(\varphi(z)), \ \ \ z \in \mathbb{D}, \ f \in H(\mathbb{D}). $$
This operator includes composition, multiplication, weighted composition, differentiation,  composition followed by differentiation operator and other operators like $D M_u$, $C_{\varphi}D$, $D C_{\varphi}$, $M_u C_{\varphi} D$, $M_u D C_{\varphi}$, $C_{\varphi} M_u D$, $D M_u C_{\varphi}$, $C_{\varphi} D M_u$ and
$D C_{\varphi} M_u$.
The extensions of this operator, called Stevi\'{c}-Sharma type operator, we use here are $ T_{\psi_1, \psi_2, \varphi}^n$  and $T_{\psi_1, \psi_2, \varphi}^{m,n}$:
$$ (T_{\psi_1, \psi_2, \varphi}^n f)(z) = \psi_1 (z) f^{(n)}(\varphi(z)) +\psi_2 (z) f^{(n+1)}(\varphi(z)), \ \ \ n \in \mathbb{N}_0  $$
$$ (T_{\psi_1, \psi_2, \varphi}^n f)(z) = \psi_1 (z) f^{(m)}(\varphi(z)) +\psi_2 (z) f^{(n)}(\varphi(z)), \ \ \ m \in \mathbb{N}_0, n \in \mathbb{N}. $$

The properties of generalized Stevi\'{c}-Sharma operator on spaces of fractional
Cauchy transforms, Hardy spaces, derivative Hardy spaces and minimal Mobius invariant spaces have been investigated in \cite{AH,ALH,H,ZG1}. The boundedness and compactness of $ T_{\psi_1, \psi_2, \varphi}^n$ from Hardy space into Zygmund-type space were studied in \cite{GL}. Guo and Mu characterized  boundedness and  essential norm of $T_{\psi_1, \psi_2, \varphi}^{m,n}$ from derivative Hardy spaces into Zygmund-type spaces in \cite{GM}. The authors of \cite{ZX}, investigated the boundedness and compactness of $ T_{\psi_1, \psi_2, \varphi}^n$  from $\mathcal{Q}_K(p,q)$ and $\mathcal{Q}_{K,0}(p,q)$ spaces to Bloch-type spaces and little Bloch-type spaces. The same research has been done for the operator $T_{\psi_1, \psi_2, \varphi}^{m,n}$ between $H^{\infty}$ and  Bloch-type spaces (little Bloch-type spaces) in \cite{ZG}.

Guo and Zhao in \cite{ZX} proved that $T_{\psi_1, \psi_2, \varphi}^n :\mathcal{Q}_K \left(p,q\right)\to {\mathcal B}_{\mu }$ is compact if and only if
$$ A(\psi_1', \varphi, \gamma + n-1 )=
A(\psi_1 \varphi' + \psi_2', \varphi, \gamma +n)=
A(\psi_2 \varphi', \varphi, \gamma + n+1)=0,  $$
where the notations are defined in \eqref{r300}. Also Zhang and Guo in \cite{ZG} proved that $T_{\psi_1, \psi_2, \varphi}^{m,n} :H^{\infty}\to {\mathcal B}_{\mu }$ , $m+1 < n$, is compact if and only if
$$ \sum_{i=1}^4 \lim_{|\varphi(z)| \rightarrow 1} \| T_{\psi_1, \psi_2, \varphi}^{m,n} f_{i, \varphi(z)} \|_{\mathcal{B}_{\mu}} = \sum_{i\in \{ m, m+1 , n, n+1 \}} \lim_{|\varphi(z)| \rightarrow 1} \frac{\mu(z) |E_i (z)|}{(1-|\varphi(z)|^2)^i} =0, $$
where the notations are defined in Section 4. Motivated by the above results, we are going to extend them and compute approximately the essential norm of $T_{\psi_1, \psi_2, \varphi}^n :\mathcal{Q}_K \left(p,q\right)\to {\mathcal B}_{\mu }$ and $T_{\psi_1, \psi_2, \varphi}^{m,n} :H^{\infty}\to {\mathcal B}_{\mu }$.
Recall that the compactness characterization is obtained from essential norm since the operator is compact if and only if the essential norm is zero.

All positive constants will be denoted by $C$ which may vary from one occurrence to another. By $A\gtrsim B$ we mean there exists a constant $C$ such that $A\geq CB$ and $A\approx B$ means that $A\gtrsim B\gtrsim A$.
\section{Essential norm of $T_{\psi_1, \psi_2, \varphi}^n :\mathcal{Q}_K \left(p,q\right)\to {\mathcal B}_{\mu }$}
Throughout this section, we assume that the following condition holds
\begin{equation}\label{GrindEQ__8_}
\int^{1}_{0} K(-\log r){\left(1-r\right)}^{\min\left\{-1,q\right\}}\left(\log \frac{1}{1-r}\right)^{\chi_{-1}(q)}rdr<\infty,
\end{equation}
where $\chi_O\left(x\right)$ is the characteristic function of the set $O$.
One of the well-known criteria for studying compact operators is stated here, see \cite{CM}.
\begin{lemma}\label{l1}
Let $\psi_1, \psi_2 \in \mathcal{H}(\mathbb{D})$, $\varphi \in \mathcal{S}(\mathbb{D})$, $n \in \mathbb{N}_0$, $p>0$, $q>-2$,  $K:[0,\infty)\to [0,\infty)$ be a nondecreasing continuous function.  Then
$T_{\psi_1, \psi_2, \varphi}^n :\mathcal{Q}_K \left(p,q\right)\to {\mathcal B}_{\mu }$ is compact  if and only if $T_{\psi_1, \psi_2, \varphi}^n :\mathcal{Q}_K \left(p,q\right)\to {\mathcal B}_{\mu }$ is bounded and  for any bounded sequence $\{f_i\}_{i=1}^\infty$ in $\mathcal{Q}_K \left(p,q\right)$ which converges to zero uniformly on compact subsets of $\mathbb{D}$, as $ i \rightarrow \infty$, we have
$\| T_{\psi_1, \psi_2, \varphi}^n f_i \|_{\mathcal{B}_{\mu}} \rightarrow 0$ as $i \rightarrow \infty$.
\end{lemma}
In the following lemma, we prove that $T_{\psi_1, \psi_2, \varphi}^n :\mathcal{Q}_K \left(p,q\right)\to {\mathcal B}_{\mu }$ is compact provided that it is bounded and
$\rho = \| \varphi \|_{\infty} = \sup_{z\in \mathbb{D}} |\varphi(z)|<1$. Note that regarding boundedness of the operator $T_{\psi_1, \psi_2, \varphi}^n :\mathcal{Q}_K \left(p,q\right)\to {\mathcal B}_{\mu }$ and applying it to the polynomials $z^n$, $z^{n+1}$ and $z^{n+2}$, we have the following facts which we use in the proof of results of the paper:
\begin{equation}\label{r201}
  \sup_{z \in \mathbb{D}} \mu(z) |\psi_1'(z)| <  \infty,
\end{equation}
\begin{equation}\label{r202}
  \sup_{z \in \mathbb{D}} \mu(z) |\psi_1(z) \varphi'(z) + \psi_2'(z)| <  \infty,
\end{equation}
\begin{equation}\label{r203}
  \sup_{z \in \mathbb{D}} \mu(z) |\psi_2(z) \varphi'(z)| <  \infty.
\end{equation}
\begin{lemma}\label{l2}
 Let $\psi_1, \psi_2 \in \mathcal{H}(\mathbb{D})$, $\varphi \in \mathcal{S}(\mathbb{D})$, $n \in \mathbb{N}_0$, $p>0$, $q>-2$,  $K:[0,\infty)\to [0,\infty)$ be a nondecreasing continuous function. If $\rho <1$ and $T_{\psi_1, \psi_2, \varphi}^n :\mathcal{Q}_K \left(p,q\right)\to {\mathcal B}_{\mu }$ is bounded, then $T_{\psi_1, \psi_2, \varphi}^n :\mathcal{Q}_K \left(p,q\right)\to {\mathcal B}_{\mu }$ is compact.
\end{lemma}
\begin{proof}
Suppose that  ${\left\{f_i\right\}}_{i}$ is a bounded sequence in  $\mathcal{Q}_K \left(p,q\right)$ such that converges to 0 uniformly on compact subsets of $\mathbb{D}$ as $i\to\infty$. According to Lemma \ref{l1}, we prove that
$\| T_{\psi_1, \psi_2, \varphi}^n f_i \|_{\mathcal{B}_{\mu}} \rightarrow 0$ as $i\rightarrow \infty$.
 The Cauchy's estimate implies that $\{f_i^{(n)}\}_{i}$, $\{f_i^{(n+1)}\}_{i}$ and $\{f_i^{(n+2)}\}_{i}$ converge uniformly  to 0 on compact subsets of $\mathbb{D}$.
Computing the norm in weighted Bloch space we get
\begin{align*}
\nonumber
\| T_{\psi_1, \psi_2, \varphi}^n f_i \|_{\mathcal{B}_\mu}=&|(T_{\psi_1, \psi_2, \varphi}^n f_i )(0)|+\sup_{z\in \mathbb{D}} \mu(z) |(T_{\psi_1, \psi_2, \varphi}^n f_i )'(z)|\\
\leq &|\psi_1 (0) f^{(n)}_{i}(\varphi(0))| + |\psi_2 (0) f^{(n+1)}_{i}(\varphi(0))| \\
 & +\sup_{z\in \mathbb{D}} \mu(z)  |\psi_1'(z) f^{(n)}_{i}(\varphi(z))|\\
& +\sup_{z\in \mathbb{D}} \mu(z) |(\psi_1(z) \varphi'(z) + \psi_2 (z))f^{(n+1)}_{i}(\varphi(z))|\\
& + \sup_{z\in \mathbb{D}} \mu(z) |\psi_2(z) \varphi'(z) f^{(n+2)}_{i}(\varphi(z))|\\
\leq &|\psi_1 (0) f^{(n)}_{i}(\varphi(0))| + |\psi_2 (0) f^{(n+1)}_{i}(\varphi(0))| \\
 & +\sup_{z\in \mathbb{D}} \mu(z)  |\psi_1'(z)| \sup_{z\in D_{\rho}}  |f^{(n)}_{i}(\varphi(z))|\\
& +\sup_{z\in \mathbb{D}} \mu(z)|\psi_1(z) \varphi'(z) + \psi_2 (z)| \sup_{z\in D_{\rho}}  |f^{(n+1)}_{i}(\varphi(z))|\\
& +\sup_{z\in \mathbb{D}} \mu(z) |\psi_2(z) \varphi'(z) | \sup_{z\in D_{\rho}}  |f^{(n+2)}_{i}(\varphi(z))|,
\end{align*}
where $D_{\rho} = \{ z \in \mathbb{D}: |\varphi| \leq \rho \}$ which is a compact subset of $\mathbb{D}$. Using \eqref{r201} to \eqref{r203}, we obtain that if  $i \rightarrow \infty$ then $\| T_{\psi_1, \psi_2, \varphi}^n f_i \|_{\mathcal{B}_\mu} \rightarrow 0$.
\end{proof}
For simplifying the calculations, we set $\gamma = \frac{q+2}{p}$ and
\begin{equation} \label{r300}
  A(u, \varphi, \gamma ) = \limsup_{|\varphi(z)|\rightarrow 1} \frac{\mu (z)|u(z)|}{(1-|\varphi(z)|^2)^{\gamma}}.
\end{equation}
\begin{theorem}
  Let $\psi_1, \psi_2 \in \mathcal{H}(\mathbb{D})$, $\varphi \in \mathcal{S}(\mathbb{D})$, $n \in \mathbb{N}_0$, $p>0$, $q>-2$,  $K:[0,\infty)\to [0,\infty)$ be a nondecreasing continuous function and $\mu$ be a normal and radial weight. If $T_{\psi_1, \psi_2, \varphi}^n :\mathcal{Q}_K (p,q)\to {\mathcal B}_{\mu }$ is bounded, then
 \begin{align} \label{r204}
 &\|T_{\psi_1, \psi_2, \varphi}^n \|_{e, \mathcal{Q}_K (p,q)\to {\mathcal B}_{\mu }} \approx \\
 &\max \{ A(\psi_1', \varphi, \gamma + n-1 ),
A(\psi_1 \varphi' + \psi_2', \varphi, \gamma +n),
A(\psi_2 \varphi', \varphi, \gamma + n+1)
\}, \nonumber
  \end{align}
\end{theorem}
\begin{proof}
  If $\rho<1$, then employing Lemma \ref{l2}, $T_{\psi_1, \psi_2, \varphi}^n :\mathcal{Q}_K \left(p,q\right)\to {\mathcal B}_{\mu }$ is compact and then $\|T_{\psi_1, \psi_2, \varphi}^n \|_{e, \mathcal{Q}_K (p,q)\to {\mathcal B}_{\mu }} =0$. On the other hand, Theorem 2 of \cite{ZX} implies that
  $$ A(\psi_1', \varphi, \gamma + n-1 ) =0, \
A(\psi_1 \varphi' + \psi_2', \varphi, \gamma +n) =0, \
A(\psi_2 \varphi', \varphi, \gamma + n+1)=0.$$
So, both sides of \eqref{r204} are equal.
Let   $\rho=1$. The proof of the lower bound is based on the choosing suitable test
functions. Let $\left\{{z }_k \right\}$ be a sequence in $\mathbb{D}$ such that   $\lim_{k\to \infty }\left|\varphi \left({z }_k\right)\right|=1$. Let
$$l_i (z) = \frac{(1-|z_k|^2)^i}{(1- \overline{z_k}z)^{\gamma + i -1}}, \ \ \ a \in \mathbb{D}, \  i=1,2,3.$$
Consider the functions
\begin{align*}
f_{z_k} (z)  = &  \frac{\gamma + n+2}{\gamma +n} \prod_{j=0}^{n-1} (\gamma +j+1)(\gamma +j+2) l_{1} (z)\\
&  -2 \frac{\gamma + n+2}{\gamma +n+1} \prod_{j=0}^{n-1} (\gamma +j)(\gamma +j+2) l_{2} (z)
+  \prod_{j=0}^{n-1} (\gamma +j)(\gamma +j+1) l_{3} (z), \\
g_{z_k} (z)  = &  \frac{\gamma + n+2}{\gamma +n+1} \prod_{j=0}^{n-1} (\gamma +j+1)(\gamma +j+2) l_{1} (z)\\
&  - \frac{2 \gamma + 2n+3}{\gamma +n+1} \prod_{j=0}^{n-1} (\gamma +j)(\gamma +j+2) l_{2} (z)
+  \prod_{j=0}^{n-1} (\gamma +j)(\gamma +j+1) l_{3} (z), \\
h_{z_k} (z)  = &   \prod_{j=0}^{n-1} (\gamma +j+1)(\gamma +j+2) l_{1} (z)\\
&  -2  \prod_{j=0}^{n-1} (\gamma +j)(\gamma +j+2) l_{2} (z)
+  \prod_{j=0}^{n-1} (\gamma +j)(\gamma +j+1) l_{3} (z).
\end{align*}
Then, one can easily check that $f_{z_k}^{(n+1)}(z_k) = f_{z_k}^{(n+2)}(z_k) =0$,
$g_{z_k}^{(n)}(z_k) = g_{z_k}^{(n+2)}(z_k) =0$, $h_{z_k}^{(n)}(z_k) = h_{z_k}^{(n+1)}(z_k) =0$ and
\begin{align*}
  f_{z_k}^{(n)}(z_k) = & \frac{2  \prod_{j=0}^{n-1} (\gamma +j)(\gamma +j+1)(\gamma +j+2)}{(\gamma + n)(\gamma +n+1)} \frac{\overline{z_k}^n}{(1-|z_k|^2)^{\gamma+n-1}}, \\
    g_{z_k}^{(n+1)}(z_k) = & - \frac{  \prod_{j=0}^{n-1} (\gamma +j)(\gamma +j+1)(\gamma +j+2)}{\gamma +n+1} \frac{\overline{z_k}^{n+1}}{(1-|z_k|^2)^{\gamma+n}}, \\
      h_{z_k}^{(n)}(z_k) = & 2  \prod_{j=0}^{n-1} (\gamma +j)(\gamma +j+1)(\gamma +j+2 )  \frac{\overline{z_k}^{n+2}}{(1-|z_k|^2)^{\gamma+n+1}}.
\end{align*}
Also the sequences $\{f_{k} \}$, $\{g_{k} \}$ and $\{h_{k} \}$ are bounded  in $\mathcal{Q}_K (p,q)$ which converge to zero uniformly on compact subsets of $\mathbb{D}$, see \cite{ZX}. For every compact operator $K : \mathcal{Q}_K (p,q) \rightarrow {\mathcal B}_{\mu }$, using the norm in weighted Bloch space, we obtain
\begin{align*}
  \| T_{\psi_1, \psi_2, \varphi}^n - K \|_{\mathcal{Q}_K (p,q) \rightarrow \mathcal{B}_{\mu}}  \succeq & \limsup_{k \rightarrow \infty}\| (T_{\psi_1, \psi_2, \varphi}^n - K)f_{\varphi(z_k)} \|_{\mathcal{B}_{\mu}} \\
  \geq & \limsup_{k \rightarrow \infty}\| T_{\psi_1, \psi_2, \varphi}^n  f_{\varphi(z_k)} \|_{\mathcal{B}_{\mu}} \\
  \geq &  \limsup_{k \rightarrow \infty} \sup_{z \in \mathbb{D}} \mu(z) | (T_{\psi_1, \psi_2, \varphi}^n  f_{\varphi(z_k)})'(z)| \\
  \geq & \limsup_{k \rightarrow \infty} \mu(z_k) | (T_{\psi_1, \psi_2, \varphi}^n  f_{\varphi(z_k)})'(z_k)| \\
  = & \limsup_{k \rightarrow \infty} \mu(z_k)  |\psi_1'(z_k) f_{\varphi(z_k)}^{(n)}(\varphi(z))|\\
& +\limsup_{k \rightarrow \infty} \mu(z_k) |(\psi_1(z_k) \varphi'(z_k) + \psi_2 (z))f_{\varphi(z_k)}^{(n+1)}(\varphi(z_k))|\\
& + \limsup_{k \rightarrow \infty} \mu(z_k) |\psi_2(z_k) \varphi'(z_k) f_{\varphi(z_k)}^{(n+1)}(\varphi(z_k))| \\
=& \limsup_{k \rightarrow \infty} \frac{2  \prod_{j=0}^{n-1} (\gamma +j)(\gamma +j+1)(\gamma +j+2)}{(\gamma + n)(\gamma +n+1)}  \\
& \times \frac{\mu(z_k)  |\psi_1'(z_k)| |\overline{\varphi(z_k)}^n|}{(1-|\varphi(z_k)|^2)^{\gamma+n-1}}.
\end{align*}
Hence
\begin{equation*}
 \| T_{\psi_1, \psi_2, \varphi}^n - K \|_{\mathcal{Q}_K (p,q) \rightarrow \mathcal{B}_{\mu}}  \succeq \limsup_{|\varphi(z)|\rightarrow 1}  \frac{\mu(z)  |\psi_1'(z)|}{(1-|\varphi(z)|^2)^{\gamma+n-1}} = A(\psi_1', \varphi, \gamma + n-1 ).
\end{equation*}
Therefore
\begin{equation} \label{r205}
  \| T_{\psi_1, \psi_2, \varphi}^n \|_{e, \mathcal{Q}_K (p,q) \rightarrow \mathcal{B}_{\mu}} = \inf_{K}  \| T_{\psi_1, \psi_2, \varphi}^n - K \|_{\mathcal{Q}_K (p,q) \rightarrow \mathcal{B}_{\mu}}
  \succeq  A(\psi_1', \varphi, \gamma + n-1 ).
\end{equation}
An analogous discussion using the sequence $\{g_{k} \}$ and $\{h_{k} \}$, the followings can be achieved:
\begin{equation*}
  \| T_{\psi_1, \psi_2, \varphi}^n \|_{e, \mathcal{Q}_K (p,q) \rightarrow \mathcal{B}_{\mu}}
  \succeq  A(\psi_1 \varphi' + \psi_2', \varphi, \gamma +n),
\end{equation*}
\begin{equation*}
  \| T_{\psi_1, \psi_2, \varphi}^n \|_{e, \mathcal{Q}_K (p,q) \rightarrow \mathcal{B}_{\mu}}
  \succeq  A(\psi_2 \varphi', \varphi, \gamma + n+1),
\end{equation*}
which along with \eqref{r205} imply the lower estimate of \eqref{r204}.

Let $\{ r_k \} \subset (0,1)$ be a sequence such that $r_k \rightarrow 1$ as $k \rightarrow \infty$.
For any  $k$, 
the operator  $T_{\psi_1, \psi_2, r_k \varphi}^n :\mathcal{Q}_K \left(p,q\right)\to {\mathcal B}_{\mu }$ is bounded, \cite{ZX}. Now from Lemma \ref{l2} we get that
$T_{\psi_1, \psi_2, r_k \varphi}^n :\mathcal{Q}_K \left(p,q\right)\to {\mathcal B}_{\mu }$ is compact. So, definition of the essential norm implies that
\begin{equation}\label{r206}
\|T_{\psi_1, \psi_2,  \varphi}^n \|_{e,\mathcal{Q}_K (p,q) \rightarrow \mathcal{B}_\mu} \leq \|T_{\psi_1, \psi_2,  \varphi}^n -T_{\psi_1, \psi_2,  r_k \varphi}^n \|_{\mathcal{Q}_K (p,q) \rightarrow \mathcal{B}_\mu}.
\end{equation}
We have
\begin{align*}
\|T_{\psi_1, \psi_2,  \varphi}^n -T_{\psi_1, \psi_2,  r_k \varphi}^n \|_{\mathcal{Q}_K (p,q) \rightarrow \mathcal{B}_\mu} = &  \sup_{\|f \|_{\mathcal{Q}_K (p,q)} \leq 1}
\|(T_{\psi_1, \psi_2,  \varphi}^n -T_{\psi_1, \psi_2,  r_k \varphi}^n)f \|_{ \mathcal{B}_\mu} \\
\leq &  \sup_{\|f \|_{\mathcal{Q}_K (p,q)} \leq 1} (
|\psi_1 (0)| | f^{(n)}(\varphi(0)) - f^{(n)}(r_k \varphi(0))| \\
&  + |\psi_2 (0)| | f^{(n+1)}(\varphi(0)) - r_k f^{(n+1)}(r_k \varphi(0))|  \\
 & +\sup_{z\in \mathbb{D}} \mu(z)  |\psi_1'(z)| |  f^{(n)}(\varphi(z))- f^{(n)}(r_k \varphi(z))|\\
& +\sup_{z\in \mathbb{D}} \mu(z) |(\psi_1(z) \varphi'(z) + \psi_2 (z))| \\
& \times  |f^{(n+1)} (\varphi(z)) -  r_k f^{(n+1)} (r_k \varphi(z))|\\
& + \sup_{z\in \mathbb{D}} \mu(z) |\psi_2(z) \varphi'(z) | | f^{(n+2)} (\varphi(z)) -  r_k^2 f^{(n+2)} (\varphi(z))|).
\end{align*}
Since the singleton $\{ \varphi(0) \}$ is compact, then as $k \rightarrow \infty$ we get
$$ |\psi_1 (0)| | f^{(n)}(\varphi(0)) - f^{(n)}(r_k \varphi(0))| \rightarrow 0, \ \ \
|\psi_2 (0)| | f^{(n+1)}(\varphi(0)) - r_k f^{(n+1)}(r_k \varphi(0))|  \rightarrow 0.$$
We divide rest of the above equation into two parts, on is taken on the set $\{ z: |\varphi(z)| \leq \delta \}$, which is compact, and the other is
on the set $\{ z: |\varphi(z)| > \delta \}$. From compactness of the set and using the equations \eqref{r201}-\eqref{r203},  the first part is tends to zero as $k \rightarrow \infty$. For the other part we have
\begin{align*}
\| & T_{\psi_1, \psi_2,  \varphi}^n -  T_{\psi_1, \psi_2,  r_k \varphi}^n \|_{\mathcal{Q}_K (p,q) \rightarrow \mathcal{B}_\mu} \\
 \leq & \sup_{\{ z: |\varphi(z)| > \delta \}} \mu(z)  |\psi_1'(z)| |  f^{(n)}(\varphi(z))- f^{(n)}(r_k \varphi(z))|\\
& +\sup_{\{ z: |\varphi(z)| > \delta \}} \mu(z) |(\psi_1(z) \varphi'(z) + \psi_2 (z))|   |f^{(n+1)} (\varphi(z)) -  r_k f^{(n+1)} (r_k \varphi(z))|\\
& + \sup_{\{ z: |\varphi(z)| > \delta \}} \mu(z) |\psi_2(z) \varphi'(z) | | f^{(n+2)} (\varphi(z)) -  r_k^2 f^{(n+2)} (\varphi(z))| \\
\leq & \sup_{\{ z: |\varphi(z)| > \delta \}} \left( \frac{\mu(z)  |\psi_1'(z)| }{(1-|\varphi(z)|^2)^{\gamma + n -1}} + \frac{\mu(z)  |\psi_1'(z)| }{(1-|r_k \varphi(z)|^2)^{\gamma + n -1}} \right) \\
& + \sup_{\{ z: |\varphi(z)| > \delta \}} \left( \frac{\mu(z) |(\psi_1(z) \varphi'(z) + \psi_2 (z))|}{(1-|\varphi(z)|^2)^{\gamma + n}} +  \frac{\mu(z) |(\psi_1(z) \varphi'(z) + \psi_2 (z))|}{(1-|r_k \varphi(z)|^2)^{\gamma + n}} \right) \\
& + \sup_{\{ z: |\varphi(z)| > \delta \}} \left( \frac{\mu(z) |\psi_2(z) \varphi'(z)|}{(1-|\varphi(z)|^2)^{\gamma + n+1}} +
 \frac{\mu(z) |\psi_2(z) \varphi'(z)|}{(1-|r_k \varphi(z)|^2)^{\gamma + n+1}} \right).
\end{align*}
Therefore
\begin{align*}
\|  T_{\psi_1, \psi_2,  \varphi}^n -  T_{\psi_1, \psi_2,  r_k \varphi}^n \|_{\mathcal{Q}_K (p,q) \rightarrow \mathcal{B}_\mu}
\leq & \sup_{\{ z: |\varphi(z)| > \delta \}} 2 \frac{\mu(z)  |\psi_1'(z)| }{(1-|\varphi(z)|^2)^{\gamma + n -1}}  \\
& + \sup_{\{ z: |\varphi(z)| > \delta \}}  2 \frac{\mu(z) |(\psi_1(z) \varphi'(z) + \psi_2 (z))|}{(1-|\varphi(z)|^2)^{\gamma + n}}  \\
& + \sup_{\{ z: |\varphi(z)| > \delta \}} 2 \frac{\mu(z) |\psi_2(z) \varphi'(z)|}{(1-|\varphi(z)|^2)^{\gamma + n+1}}.
\end{align*}
As $\delta \rightarrow 1$, we obtain
\begin{align*}
\|  T_{\psi_1, \psi_2,  \varphi}^n -  T_{\psi_1, \psi_2,  r_k \varphi}^n \|_{\mathcal{Q}_K (p,q) \rightarrow \mathcal{B}_\mu}
\leq & \limsup_{|\varphi(z)| \rightarrow 1} 2 \frac{\mu(z)  |\psi_1'(z)| }{(1-|\varphi(z)|^2)^{\gamma + n -1}}  \\
& + \limsup_{|\varphi(z)| \rightarrow 1}  2 \frac{\mu(z) |(\psi_1(z) \varphi'(z) + \psi_2 (z))|}{(1-|\varphi(z)|^2)^{\gamma + n}}  \\
& + \limsup_{|\varphi(z)| \rightarrow 1} 2 \frac{\mu(z) |\psi_2(z) \varphi'(z)|}{(1-|\varphi(z)|^2)^{\gamma + n+1}}.
\end{align*}
Hence, \eqref{r206} implies that
 \begin{align*}
 \|T_{\psi_1, \psi_2, \varphi}^n \|_e \preceq
   A(\psi_1', \varphi, \gamma + n-1 )+
A(\psi_1 \varphi' + \psi_2', \varphi, \gamma +n)+
A(\psi_2 \varphi', \varphi, \gamma + n+1),
  \end{align*}
and the proof is completed.
\end{proof}
\section{Essential norm of $T_{\psi_1, \psi_2, \varphi}^{m,n} :H^{\infty}\to {\mathcal B}_{\mu }$}
\begin{lemma}\cite{ZG} \label{l4}
Let $\psi_1, \psi_2 \in \mathcal{H}(\mathbb{D})$, $\varphi \in \mathcal{S}(\mathbb{D})$, $m \in \mathbb{N}_0$, $n \in \mathbb{N}$ and  $\mu$ be a radial weight such that the operator  $T_{\psi_1, \psi_2, \varphi}^{m,n} :H^{\infty}\to {\mathcal B}_{\mu }$ is bounded. Then
$T_{\psi_1, \psi_2, \varphi}^{m,n} :H^{\infty}\to {\mathcal B}_{\mu }$ is compact  if and only if $ \| T_{\psi_1, \psi_2, \varphi}^{m,n} \|_{\mathcal{B}_{\mu}} \rightarrow 0$ as $k \rightarrow \infty$ for each bounded sequence $\{ f_k \}_{k \in \mathbb{N}}$ in $H^{\infty}$ which converges to zero uniformly on compact subsets of $\mathbb{D}$ as $k \rightarrow \infty$.
\end{lemma}
Before stating the following theorem, note that since polynomials belong to $H^{\infty}$ then the boundedness of the operator $T_{\psi_1, \psi_2, \varphi}^{m,n} :H^{\infty}\to {\mathcal B}_{\mu }$ implies that
\begin{equation*}
  \sum_{i \in \{ m, m+1 , n, n+1 \}} \sup_{z \in \mathbb{D}} \mu(z) |E_i (z)| < \infty,
\end{equation*}
where
\begin{align*}
E_m (z) = & \psi_1'(z), \ \ \ E_{m+1} (z) = \psi_1 (z) \varphi'(z) \\
E_n (z) = & \psi_2'(z), \ \ \ E_{n+1} (z) = \psi_2 (z) \varphi'(z).
\end{align*}
\begin{theorem}
  Let $\psi_1, \psi_2 \in \mathcal{H}(\mathbb{D})$, $\varphi \in \mathcal{S}(\mathbb{D})$, $m \in \mathbb{N}_0$, $n \in \mathbb{N}$, $m+1 < n$ and  $\mu$ be a radial weight. Suppose that
$T_{\psi_1, \psi_2, \varphi}^{m,n} :H^{\infty}\to {\mathcal B}_{\mu }$ is bounded. Then
\begin{align*}
  \| T_{\psi_1, \psi_2, \varphi}^{m,n} \|_{e, H^{\infty}\to {\mathcal B}_{\mu}}  \approx &  \max_{i \in \{ m, m+1 , n, n+1 \}} \left \{ \limsup_{|\varphi(z)| \rightarrow 1} \frac{\mu(z) |E_i (z)|}{(1-|\varphi(z)|^2)^i} \right \}  \\
  \approx &
  \max_{1 \leq i \leq 4}  \left \{ \limsup_{|\varphi(w)| \rightarrow 1} \| T_{\psi_1, \psi_2, \varphi}^{m,n} f_{i, \varphi(w)} \|_{\mathcal{B}_{\mu}} \right \},
\end{align*}
where $f_{i, \varphi(w)}(z) = \left( \frac{1-|\varphi(w)|}{1-\overline{\varphi(w)}z}\right)^i$.
\end{theorem}
\begin{proof}
The first part of the theorem will be proved, the other is the same. Let $\left\{{w}_k \right\}$ be a sequence in $\mathbb{D}$ such that   $\lim_{k\to \infty }\left|\varphi \left({w}_k\right)\right|=1$. According to Lemma 1 of \cite{ZG}, there exist the functions $g_{i,w_k} \in H^{\infty}$ such that
$$ g_{i,w_k}^{j} = \frac{\overline{w_k}^j \delta_{ij}}{(1-|w_k|^2)^j}, \ \ \ i,j \in \{ m, m+1 , n, n+1 \}.  $$
Also the sequence $\{ g_{i, w_k} \}$ is bounded in $H^{\infty}$ and converges to 0 uniformly on compact subsets of $\mathbb{D}$.
So for every compact operator $K :  H^{\infty} \rightarrow {\mathcal B}_{\mu }$, $\lim_{k \rightarrow \infty} \| K g_{i,w_k} \|_{\mathcal{B}_{\mu}} =0$.  using the norm in weighted Bloch space, we obtain
\begin{align*}
  \| T_{\psi_1, \psi_2, \varphi}^{m,n} - K \|_{H^{\infty} \rightarrow \mathcal{B}_{\mu}}  \succeq & \limsup_{k \rightarrow \infty}\| (T_{\psi_1, \psi_2, \varphi}^{m,n} - K)g_{i, \varphi(w_k)} \|_{\mathcal{B}_{\mu}} \\
  \geq & \limsup_{k \rightarrow \infty}\| T_{\psi_1, \psi_2, \varphi}^{m,n}  g_{i, \varphi(w_k)} \|_{\mathcal{B}_{\mu}} \\
  \geq &  \limsup_{k \rightarrow \infty} \sup_{z \in \mathbb{D}} \mu(z) | (T_{\psi_1, \psi_2, \varphi}^{m,n}  g_{i, \varphi(w_k)})'(z)| \\
  \geq & \limsup_{k \rightarrow \infty} \mu(w_k) | (T_{\psi_1, \psi_2, \varphi}^{m,n}  g_{i, \varphi(w_k)})'(w_k)| \\
  = & \limsup_{k \rightarrow \infty} \sum_{j \in \{ m, m+1 , n, n+1 \}} \mu(w_k)
  |E_j (w_k) (w_k) g_{i, \varphi(w_k)}^{(j)} (\varphi(w_k))| \\
  = & \limsup_{k \rightarrow \infty} \frac{\mu(w_k) |E_{i} (w_k) \overline{\varphi(w_k)}^i| }{(1-|\varphi(w_k)|^2)^i} \\
  =& \limsup_{k \rightarrow \infty} \frac{\mu(w_k) |E_{i} (w_k)| }{(1-|\varphi(w_k)|^2)^i}.
\end{align*}
Therefore
\begin{equation} \label{r208}
  \| T_{\psi_1, \psi_2, \varphi}^{m,n} \|_{H^{\infty} \rightarrow \mathcal{B}_{\mu}} = \inf_{K}  \| T_{\psi_1, \psi_2, \varphi}^{m,n} - K \|_{H^{\infty} \rightarrow \mathcal{B}_{\mu}}
  \succeq  \limsup_{|\varphi(z)| \rightarrow 1} \frac{\mu(z) |E_i (z)|}{(1-|\varphi(z)|^2)^i},
\end{equation}
for every $i \in \{ m, m+1 , n, n+1 \}$.

Now, let $\{ r_k \} \subset (0,1)$ be a sequence such that $r_k \rightarrow 1$ as $k \rightarrow \infty$. Since $T_{\psi_1, \psi_2, \varphi}^{m,n} :H^{\infty}\to {\mathcal B}_{\mu }$ is bounded, we have
$$ \sum_{j \in \{ m, m+1 , n, n+1 \}} \sup_{z \in \mathbb{D}} \frac{\mu (z) |E_j (z)|}{(1-|\varphi(z)|^2)^j} < \infty, $$
see \cite{ZG}, Theorem 1. It is easy to see that if we replace $r_k \varphi$ by $\varphi$, then we have
$$ \sum_{j \in \{ m, m+1 , n, n+1 \}} \sup_{z \in \mathbb{D}} \frac{\mu (z) |E_j (z)|}{(1-|r_k \varphi(z)|^2)^j} < \infty,  $$
which implies that $T_{\psi_1, \psi_2, r_k \varphi}^{m,n} :H^{\infty}\to {\mathcal B}_{\mu}$ be  bounded. So it is compact. Then
\begin{equation}\label{r209}
\|T_{\psi_1, \psi_2,  \varphi}^{m,n} \|_{e,H^{\infty} \rightarrow \mathcal{B}_\mu} \leq \|T_{\psi_1, \psi_2,  \varphi}^{m,n} -T_{\psi_1, \psi_2,  r_k \varphi}^{m,n} \|_{H^{\infty} \rightarrow \mathcal{B}_\mu}.
\end{equation}
For simplifying the calculations, we set $f_r (z) = f(rz)$, $0 < r < 1$. Then we have
\begin{align*}
\|T_{\psi_1, \psi_2,  \varphi}^{m,n} -T_{\psi_1, \psi_2,  r_k \varphi}^{m,n} \|_{H^{\infty} \rightarrow \mathcal{B}_\mu} = &  \sup_{\|f \|_{H^{\infty}} \leq 1}
\|(T_{\psi_1, \psi_2,  \varphi}^{m,n} -T_{\psi_1, \psi_2,  r_k \varphi}^{m,n})f \|_{ \mathcal{B}_\mu} \\
\leq &  \sup_{\|f \|_{H^{\infty}} \leq 1} (
|\psi_1 (0)| | f^{(m)}(\varphi(0)) - f^{(m)}(r_k \varphi(0))| \\
&  + |\psi_2 (0)| | f^{(n)}(\varphi(0)) - r_k f^{(n)}(r_k \varphi(0))|)   \\
 & +  \sup_{\|f \|_{H^{\infty}} \leq 1} \sum_{i \in \{ m, m+1 , n, n+1 \}} \sup_{z\in \mathbb{D}}  \mu(z) |E_i (z) (f-f_{r_k})^{(i)}(\varphi(z))|.
\end{align*}
Letting $k \rightarrow \infty$ we get
$$ |\psi_1 (0)| | f^{(m)}(\varphi(0)) - f^{(m)}(r_k \varphi(0))| \rightarrow 0, \ \ \
|\psi_2 (0)| | f^{(n)}(\varphi(0)) - r_k f^{(n)}(r_k \varphi(0))|  \rightarrow 0.$$
Also
\begin{align*}
 \sum_{i \in \{ m, m+1 , n, n+1 \}} &  \sup_{z\in \mathbb{D}}  \mu(z) |E_i (z) (f-f_{r_k})^{(i)}(\varphi(z))| \\
  \leq &  \sum_{i \in \{ m, m+1 , n, n+1 \}} \sup_{|\varphi(z)| \leq \delta}  \mu(z) |E_i (z) (f-f_{r_k})^{(i)}(\varphi(z))| \\
 & +  \sum_{i \in \{ m, m+1 , n, n+1 \}} \sup_{|\varphi(z)| > \delta}  \mu(z) |E_i (z) (f-f_{r_k})^{(i)}(\varphi(z))|
\end{align*}
and letting $k \rightarrow \infty$ we obtain
$$ \sum_{i \in \{ m, m+1 , n, n+1 \}} \sup_{|\varphi(z)| \leq \delta}  \mu(z) |E_i (z) (f-f_{r_k})^{(i)}(\varphi(z))| \rightarrow 0. $$
Moreover for every $f \in H^{\infty}$, $\| f \|_{\infty} \leq 1$ we get
\begin{align*}
 \sum_{i \in \{ m, m+1 , n, n+1 \}} & \sup_{|\varphi(z)| > \delta}  \mu(z) |E_i (z) (f-f_{r_k})^{(i)}(\varphi(z))| \\
 \leq &  \sum_{i \in \{ m, m+1 , n, n+1 \}} \sup_{|\varphi(z)| > \delta}  \mu(z) |E_i (z)| |  f^{(i)}(\varphi(z))| \\
 & + \sum_{i \in \{ m, m+1 , n, n+1 \}} \sup_{|\varphi(z)| > \delta}  \mu(z) |E_i (z)| | r_k^{i} f^{(i)}(r_k\varphi(z))| \\
 \preceq & \sum_{i \in \{ m, m+1 , n, n+1 \}} \sup_{|\varphi(z)| > \delta} \frac{\mu(z) |E_i (z)|}{(1-|\varphi(z)|^2)^i } \| f \|_{\infty} \\
 & + \sum_{i \in \{ m, m+1 , n, n+1 \}} \sup_{|\varphi(z)| > \delta} \frac{\mu(z) r_k^i  |E_i (z)|}{(1-|r_k \varphi(z)|^2)^i } \| f \|_{\infty} \\
 \preceq & \sum_{i \in \{ m, m+1 , n, n+1 \}} \sup_{|\varphi(z)| > \delta} \frac{\mu(z) |E_i (z)|}{(1-|\varphi(z)|^2)^i }.
\end{align*}
If $\delta \rightarrow 1$, then
\begin{equation*}
  \|T_{\psi_1, \psi_2,  \varphi}^{m,n} -T_{\psi_1, \psi_2,  r_k \varphi}^{m,n} \|_{H^{\infty} \rightarrow \mathcal{B}_\mu}  \preceq  \sum_{i \in \{ m, m+1 , n, n+1 \}} \limsup_{|\varphi(z)| \rightarrow 1} \frac{\mu(z) |E_i (z)|}{(1-|\varphi(z)|^2)^i}.
\end{equation*}
Finally
\begin{equation}\label{r211}
\|T_{\psi_1, \psi_2,  \varphi}^{m,n} \|_{e,H^{\infty} \rightarrow \mathcal{B}_\mu}
\preceq  \sum_{i \in \{ m, m+1 , n, n+1 \}} \limsup_{|\varphi(z)| \rightarrow 1} \frac{\mu(z) |E_i (z)|}{(1-|\varphi(z)|^2)^i}.
\end{equation}
\end{proof}
By using the same method we can prove the following theorem.
\begin{theorem}
  Let $\psi_1, \psi_2 \in \mathcal{H}(\mathbb{D})$, $\varphi \in \mathcal{S}(\mathbb{D})$, $m \in \mathbb{N}_0$, $n \in \mathbb{N}$, $m+1 = n$ and  $\mu$ be a radial weight. Suppose that
$T_{\psi_1, \psi_2, \varphi}^{m,n} :H^{\infty}\to {\mathcal B}_{\mu }$ is bounded. Then
\begin{align*}
  \| T_{\psi_1, \psi_2, \varphi}^{m,n} \|_{e, H^{\infty}\to {\mathcal B}_{\mu}}  \approx & \max_{i \in \{ m, n+1 \}} \left \{ \limsup_{|\varphi(z)| \rightarrow 1} \frac{\mu(z) |E_i (z)|}{(1-|\varphi(z)|^2)^i}, \limsup_{|\varphi(z)| \rightarrow 1} \frac{\mu(z) |\psi_1 (z) \varphi'(z) + \psi_2'(z)|}{(1-|\varphi(z)|^2)^n} \right \} \\
  \approx &   \max_{1 \leq i \leq 3}  \{ \limsup_{|\varphi(w)| \rightarrow 1} \| T_{\psi_1, \psi_2, \varphi}^{m,n} f_{i, \varphi(w)} \|_{\mathcal{B}_{\mu}}  \}.
\end{align*}
\end{theorem}


\begin{thebibliography}{99}
\bibitem{AH}
E. Abbasi and M. Hassanlou, Generalized Stevi\'{c}-Sharma type operators on spaces of fractional Cauchy transforms, Mediterr. J. Math., 21(1) (2024), pp. 40.

\bibitem{ALH}
E. Abbasi, Y.  Liu and M. Hassanlou, Generalized Stevi\'c-Sharma type
operators from Hardy spaces into $n$th weighted type spaces,
Turkish Journal of Mathematics 45(4) (2021), Article 4.  https://doi.org/10.3906/mat-2011-67

\bibitem{CM}
C. C. Cowen and B. D. Maccluer, Composition operators on spaces of analytic functions, Studies in Advanced Mathematics. CRC Press, Boca Raton, 1995.


\bibitem{GM}
Z. Guo and J.  Mu, Generalized Stevi\'{c}-Sharma type operators from derivative Hardy spaces into Zygmund-type spaces, AIMS Mathematics, 8(2) (2022), 3920-3939.

\bibitem{ZX}
Z. Guo and X. Zhao, On a Stevi\'{c}-Sharma type operator from $\mathcal{Q}_K(p,q)$ spaces to Bloch-type spaces, Oper. Matrices, 16 (2022), 563-580. http://dx.doi.org/10.7153/oam-2022-16-43

\bibitem{GL}
Z. Guo and L. Liu, Product-type operators from Hardy spaces to Bloch-type spaces and Zygmund -type spaces, Numer. Funct. Anal. Optim. 43(10) (2022), 1240-1264

\bibitem{H}
M. Hassanlou, E. Abbasi, M. K. Arpatapeh and S. Nasresfahani,
Product-type operators between minimal Mobius invariant spaces
and Zygmund type spaces, Sahand Commun. Math. Anal., 20(3)
(2023),  69–50.

\bibitem{kot}
M. Kotilainen,  On composition operators in $\mathcal{Q}_K $ type spaces. J. Funct. Space Appl. 5 (2007), 103-122. https://doi.org/10.1155/2007/956392

\bibitem{LS1}
 S. Li and S. Stevi\'{c}, Products of Volterra type operator and composition operator from $H^\infty$ and Bloch spaces to the Zygmund space, J. Math. Anal. Appl. 345 (2008), 40-52.

\bibitem{Liu}
Y. Liu and Y. Yu, Generalized integration operators from $\mathcal{Q}_K(p,q)$ to the little Zygmund-type spaces, New Trends in Analysis and Interdisciplinary Applications, DOI 10.1007/978-3-319-48812-7\_30.

\bibitem{MHV}
A. Manavi, M. Hassanlou and H. Vaezi, Essential norm of generalized integral type operator from $\mathcal{Q}_{K}(p,q)$ to Zygmund spaces, Filomat 37(16) (2023), 5273-5282.

\bibitem{M}
X. Meng, Some sufficient conditions for analytic functions to belong to  $\mathcal{Q}_{K,0}(p,q)$ space, Abstr. Appl. Anal. 2008 (2008), Article ID 404636.

\bibitem{P}
C. Pan, On an integral-type operator from $\mathcal{Q}_K(p,q)$ spaces to $\alpha$-Bloch space, Filomat 25 (2011), 163-173.

\bibitem{Ren}
Y. Ren, An integral-type operator from $\mathcal{Q}_K(p,q)$ spaces to Zygmund-type spaces, Appl. Math. Comput. 236 (2014), 27-32.

\bibitem{SS1}
S. Stevi\'{c} , A. K. Sharma and A. Bhat, Essential norm of products of multiplication composition and
differentiation operators on weighted Bergman spaces, Appl. Math. Comput. 218(6) (2011), 2386-2397.

\bibitem{SS2}
S. Stevi\'{c} , A. K. Sharma and A. Bhat, Products of multiplication composition and differentiation
operators on weighted Bergman space, Appl. Math. Comput. 217(20) (2011), 8115-8125.

\bibitem{tj}
M. Tjani, Compact composition operators on some Mobius invariant Banach space, PhD dissertation, Michigan State university, Michigan, USA, 1996.

\bibitem{Wz}
H. Wulan and J. Zhou, $\mathcal{Q}_K$ type spaces of analytic functions, J. Funct. Spaces Appl. 4 (2006), 73-84.

\bibitem{Y}
W. Yang, Products of composition and differentiation operators from $\mathcal{Q}_K(p,q)$ spaces to Bloch-type spaces, Abstr. Appl. Anal. 2009 (2009) Article ID 741920.

\bibitem{ZG}
Q. Zhang and Z. Guo, Generalized Stevi\'{c}-Sharma type operators from $H^{\infty}$ space into Bloch-type spaces,  Math. Ineq. Appl. 26(2) (2023), 531-543.

\bibitem{ZG1}
F. Zhang and Y. Liu, On a Stevi\'{c}-Sharma operator from Hardy spaces to Zygmund-type spaces on the unit disk, Complex Anal. Oper. Theory. 12(1) (2018), 81–100.

\bibitem{zhao}
R. Zhao, On $F(p,q,s)$ spaces, Acta Math. Scientia, 41B(6) (2021), 1985-2020.
https://doi.org/10.1007/s10473-021-0613-3

\bibitem{zhu1}
K. Zhu, Bloch type spaces of analytic functions, Rocky Mt. J. Math. 23 (1993),  1143-1177.








\end{thebibliography}
\end{document}